\documentclass{amsart}
\usepackage{euscript}           
\usepackage{amsmath,amsthm}     
\usepackage{amssymb}            
\usepackage{graphicx}

\newlength{\labwidth}
\newcommand{\labarrow}[1]{
\settowidth{\labwidth}{$\scriptstyle \;\; #1 \;\;$}
\stackrel{#1}{\smash{\hbox to \labwidth{\rightarrowfill}}
\vphantom{\longrightarrow}}
}

\newcommand{\QQ}{{\mathbf Q}}

\usepackage {epsf}
\DeclareMathOperator*{\colim}{colim}

\newcommand {\3}{{$\Gamma_1(3)$}}

\newcommand {\ZZ}{{\mathbb Z}}

\newcommand{\ra}{\rightarrow}                   
\newcommand{\lra}{\longrightarrow}              

\newfont{\german}       {eufm10 at 12pt}
\DeclareMathOperator{\Hom}{Hom} \DeclareMathOperator{\Ext}{Ext}

\DeclareMathOperator{\Tor}{Tor}
\newtheorem{thm}{Theorem}[section]
\newcounter{numerierer}
\newcounter{leer}

\newtheorem{defn}[thm]{Definition}
\newtheorem{prop}[thm]{Proposition}

\newtheorem{cor}[thm]{Corollary}
\newtheorem{lemma}[thm]{Lemma}

\theoremstyle{definition}  
\newenvironment{definition}{\begin{defn}\rm}{\end{defn}}

\newtheorem{conj}[thm]{Conjecture}

\newtheorem{remark}[thm]{Remark}
\usepackage[frame,matrix,curve, arrow]{xy}

\xymatrixcolsep{1.9pc}                          
\xymatrixrowsep{1.9pc}
\newdir{ >}{{}*!/-5pt/\dir{>}}                  
\newdir{ |}{{}*!/-5pt/\dir{|}}                  

\title{Characteristic classes in $TMF$ of level \3 }
\author{Gerd Laures}
\address{ Fakult\"at f\"ur Mathematik,  Ruhr-Universit\"at Bochum, NA1/66,
  D-44780 Bochum, Germany}
\begin{document}
\begin{abstract}
Let $TMF_1(n)$ be the spectrum of topological modular forms equipped with a $\Gamma_1(n)$-structure. We compute the $K(2)$-local $TMF_1(3)$-cohomology of $B{\mathit String}$ and $B{\mathit Spin}$: both are power series rings freely generated by classes that we explicitly construct and which generalize the classical Pontryagin classes. As a first application of this computation, we show how to construct $TMF(3n)$-cohomology classes from stable positive energy representations of the loop groups $L{\mathit Spin}$. 
\end{abstract}
\maketitle
\section{Introduction and Statement of Results}
Characteristic classes are cohomology classes which are naturally associated to principal $G$-bundles on topological spaces. Examples include Stiefel-Whitney classes for $O(n)$-bundles or $KO$-Pontryagin classes for $SO(n)$-bundles. If the cohomology theory is oriented with respect to $G$-bordism, the characteristic classes can be evaluated on fundamental classes of  manifolds with $G$-structure and yield characteristic numbers. It is well known that oriented bordism is determined by  Stiefel-Whitney numbers and rational Pontryagin numbers, whereas spin bordism is determined by Stiefel-Whitney numbers and $KO$-Pontryagin numbers.\par
In this paper we consider characteristic classes for spin and string bundles. Here, $String$ is the 6-connected cover of the $Spin$-group. The cohomology theory is $TMF_1(3)$, the cohomology of topological modular forms of level \3. It is possible that  string bordism is determined by Stiefel-Whitney  and $TMF$-characteristic numbers for some level structures. At the prime 2, the connective version of $TMF$ should split off of the Thom spectrum $MString$ and there is some evidence that another summand is given by the 16th suspension of the connective cover of $TMF_0(3)$ (compare \cite{MR2508904}). However, in order to get a map, a better understanding of the characteristic classes is necessary. (The relation between $TMF_1(3)$ and $TMF_0(3)$ is described in section \ref{tmf}.) \par
The homology groups $E_*BU\langle 6 \rangle $ for complex oriented theories $E$ have been described in \cite{MR1869850} in terms of cubical structures on the formal group of $E$. Here, $BU\langle 6\rangle$ is the complex analogue of $B{\mathit String}$. One may hope that the real case $E_*B{\mathit String}$ admits a similar description in terms of ``real'' cubical structures. In \cite{LO14} the relation between cubical structures and the results of the current article has been further investigated.\par
The cohomology groups $E^*B{\mathit String}$ for various complex oriented versions of $TMF$ have been considered in \cite{MR1660325}. When $2$ is inverted these groups have been computed by the topological $q$-expansion principle. The difficulty appears at the prime $2$ since here a splitting result for the string bordism  spectrum along the lines of Hovey-Ravenel \cite{MR1297530} is missing. This paper deals with this remaining prime. 
\par
We first look at characteristic classes for spin bundles and show the following $q$-expansion principle.
\begin{thm}\label{main}
The diagram
$$
\xymatrix{TMF_1(3)^*B{\mathit Spin} \ar[r]^\lambda \ar[d] &K_{Tate}^*B{\mathit Spin}\ar[d]^{ch}\\
 H^*(B{\mathit Spin}, TMF_1(3)^*_\QQ)\ar[r]&  H^*(B{\mathit Spin}, {K_{Tate}}^*_\QQ)}
 $$
 is a pullback.
\end{thm}
The horizontal map is Miller's elliptic character which corresponds to the evaluation at the Tate curve on the moduli stack of elliptic curves (\cite{MR1022688}.) On coefficients this map is just the traditional q-expansion map for modular forms. The theory $K_{Tate}$ is the power series ring $K[1/3]((q))$ of $K$-theories as in \cite{MR1660325}. The right vertical map is the Chern character and the left vertical map is the Dold character, that is, the map to rational cohomology induced by the exponential of the formal group law (see ibid.)
\par
The theorem determines the ring of $TMF_1(3)$-characteristic classes for spin bundles as follows. An element of $K_{Tate}^*B{\mathit Spin}$ is a $K_{Tate}$-characteristic class for spin bundles, that is, a formal series of virtual vector bundles which is naturally defined for spin bundles. If its  Chern character is invariant under the appropriate M\"obius transformations then it gives rise to a unique  $TMF_1(3)$-characteristic class.
This property allows the construction of many natural classes. An application to  representations  of loop groups will be given later.
 \par

The proof of the theorem is given in Section \ref{proof thm}.  Locally at the prime 2 the cohomology theory $TMF_1(3)$ is a generalized Johnson-Wilson spectrum $E(2)$ by Proposition \ref{ge2}. Hence it relates to the Morava $K$-theory $K(2)$ by a Bockstein spectral sequence. A splitting principle which is based on the Kitchloo-Laures computation  \cite{MR1909866} of the Morava $K(2)$-homology of $B{\mathit Spin}$ forces  the universal coefficient spectral sequence to collapse. 
Its $E_2 $-term is studied via a chromatic resolution. 
Note that this part of the work does not involve the theory of elliptic curves but only the study of formal group laws of height at most 2. However, the statement of the theorem is a transchromatic property which finds a natural formulation in the $q$-expansion of modular forms once the generalized $E(2)$-theory has been chosen to be $TMF_1(3)$. 
The pullback property is reduced via the universal coefficient isomorphisms to the classical q-expansion principle.
\par
The theorem allows us to the construct Pontryagin classes for $TMF_1(3)$. The pullback diagram relates these to the classical Pontryagin classes of $K$-theory and rational Pontryagin classes in a nice way.
\begin{thm}\label{Pontryagin}
\begin{enumerate}
\item
There are unique classes $p_i\in TMF_1(3)^{4i}B{\mathit Spin}$ with the following property: the formal series  $p(t)=1+p_1t+p_2t^2\ldots$ is given by $$ \prod_{i=1}^m (1-t\rho^*(x_i \overline{x}_i))$$
when restricted to the classifying space of each maximal torus of $Spin(2m)$. Here, $\rho$ is the map to the maximal torus of $SO(2m)$ and the $x_i$ (and  $\overline{x}_i$) are the first $TMF_1(3)$-Chern classes of the canonical line bundles $L_i$ (resp.\ $\overline{L}_i$) over the classifying spaces of the tori. 
\item
The classes $p_i$ freely generate  the $TMF_1(3)$-cohomology of $B{\mathit Spin}$, that is, 
$$ TMF_1(3)^*B{\mathit Spin}\cong TMF_1(3)^*[\! [p_1,p_2,\ldots ]\!].$$

\end{enumerate}
\end{thm}
Next we turn to the calculation of the characteristic classes for $B{\mathit String}$. Here it turns out that the $q$-expansion principle only holds for odd primes. A string characteristic class may not even be determined by its characters. However, we have the following result.
\begin{thm}\label{BString}
Let $\widehat{TMF}_1(3)$ be the $K(2)$-localization of $TMF_1(3)$ at the prime 2. Then there is an isomorphism of algebras
$$ \widehat{TMF}_1(3)^*B{\mathit String}\cong \widehat{TMF}_1(3)^*[\! [r,p_1,p_2,\ldots ]\!] $$
where $p_1,p_2,\ldots$ are the Pontryagin classes coming from $B{\mathit Spin}$ and $r$ restricts to a topological generator of degree 6 in the $K(2)$-cohomology of $K(\ZZ, 3)$. 
\end{thm}
Theorems  \ref{Pontryagin} and \ref{BString} solve the problem stated in \cite{MR2093483}:  they give  explicit generators for the cohomology rings of $B{\mathit Spin}$ and $B{\mathit String}$.  Their proofs  are given in Section \ref{SecBString}. Theorem \ref{BString} relies on a 
description of the Morava $K(2)$-homology of $B{\mathit String}$  given by  Kitchloo-Laures-Wilson in \cite{MR2093483, MR2086079}. We remark that  the completed $TMF_1(3)$ can be described as a fixed point spectrum of Lubin-Tate theory (see \cite{MR2508904}.) 
\par
As a first consequence we give a proof of the stable version of a conjecture by Brylinski which was stated  in \cite{MR1071369}. The precise formulation is given in the last section. For sufficiently large $n$ divisible by 3 as in Theorem \ref{coxeter}, there is a group homomorphism 
$$\varphi: (P_m)_{\Gamma(n)}\lra {TMF}(n )^*B{\mathit String}$$
 from a stable group of positive energy representations V of the  free loop group ${L}Spin$ to the  $TMF$-cohomology with level $n$-structure. Here, the congruence group is larger than the one considered before since the character of the representation is only known to be invariant under the action of $\Gamma(n)$ by a theorem of Kac and Wakimoto \cite[Theorem A]{MR954660}. \par
We will describe the map $\varphi$ in terms of its elliptic character. Suppose $P$ is a $String$-principal bundle over $X$. Let $\tilde{L}Spin$ be the universal central extension of $L{\mathit Spin}$.  Then by \cite{MR970288}, $LX$ carries a $\tilde{L}Spin$-principal bundle $\tilde{L}P$ whose associated $L{\mathit Spin}$-bundle is $LP$. In particular, this holds for the universal $String$-bundle $E{\mathit String}$  over $B{\mathit String}$. The elliptic character of $\varphi$ gives the bundle 
$$ \lambda \varphi(V)=(\tilde{L} E Spin\times_{\tilde{L}Spin}V)_{|B{\mathit String}}$$
(when $V$ is suitably normalized with a character of the rotation circle).
In this formula, the right hand side is considered as a formal power series of virtual bundles by decomposing the bundle as a representation of the circle group which reparameterizes the loops.
The evaluation of this class on the fundamental class of a string manifold is the formal index of the Dirac operator on $LM$ with coefficient in the bundle associated to the representation $V$ (see \cite{MR1071369}).
\subsubsection*{Acknowledgements.} The author would like to thank  Nitu Kitchloo, Martin Olbermann, Bj\"orn Schuster, Vesna Stojanoska, Neil Strickland and Steve Wilson for helpful discussions. He is also grateful to the referee for a very careful revision.

\section{Topological modular forms of level \3}\label{tmf}
In this section we describe the spectrum of topological modular forms with level structures. Most of the results are well known and  are collected for later reference. We start more generally than actually needed in order to put the main theorem in a broader picture.\par
Let ${\mathcal M}$ denote the stack of smooth elliptic curves. A morphism $f:S\lra {\mathcal M}$ determines an elliptic curve $C_f$ over the base scheme $S$ (see Deligne and Rapoport  \cite{MR0337993}.) The cotangent spaces of the identity section of each elliptic curve give rise to the line bundle $\omega$ of invariant differentials on ${\mathcal M}$.  A section of $\omega^{\otimes k}$ is called a modular form of weight $k$. We write $M_k$ for the group of modular forms of weight $k$. 
\par
Let $C$ be an elliptic curve over a ring $R$. Let $C[n]$ denote the kernel of the self-map which multiplies by $n$ on $C$. If $n$ is invertible in $R$, then \'{e}tale-locally, $C[n]$ is of the form $\ZZ/n\times \ZZ/n$. A choice of an isomorphism is called a $\Gamma(n)$-structure on $C$. A monomorphism $\ZZ/n\lra C[n]$ is a $\Gamma_1(n)$-structure. It corresponds to a choice of a point of exact order $n$. A choice of a subgroup scheme of $C[n]$ which is isomorphic to $\ZZ/n$ is called a $\Gamma_0(n)$-structure. There are maps of associated moduli stacks with $n$ inverted
\begin{eqnarray}\label{level}
&\xymatrix{ {\mathcal M}_{\Gamma(n)}\ar[r]& {\mathcal M}_{\Gamma_1(n)}\ar[r]& {\mathcal M}_{\Gamma_0(n)}\ar[r]& {\mathcal M}}&
\end{eqnarray}
and the sections of the associated  line bundles are called modular forms with the corresponding level structures. For example, a $\Gamma_1(n)$-modular form $f$ of weight $k$ associates an element of the ring $R$ to each triple $(C/R, \omega, P)$ where $C$ is an elliptic curve over $R=R[1/n]$, $\omega$ is a translation invariant nowhere vanishing differential and $P$ is a point of exact order $n$. This association should only depend on the isomorphism class of the triple, it should be invariant under base change, and  it should satisfy
$$ f(C,a \omega, P)=a^{-k} f(C, \omega, P)$$
 for all $a\in R^{\times}$. \par
 The case $n=3$ can be made more explicit. Locally, for any such triple $(C,\omega, P)$ the curve can uniquely be written in the form 
 \begin{eqnarray}\label{curve}
 & C: &y^2 +a_1xy +a_3 y = x^3
 \end{eqnarray}
 in a way that $P$ is the origin $(0,0)$ and $\omega$ has the standard form $\omega=dx/(2y+a_1x+a_3)$. A proof of this fact can be found  in \cite[3.2]{MR2508904}. It means that there is universal triple $(C,\omega, (0,0))$ over the ring $\ZZ[1/3, a_1,a_3,\Delta^{-1}]$ with the property that locally any other triple is obtained by base change. Hence a $\Gamma_1(3)$-modular form is determined by its value on the universal triple and gives an element in this ring. Moreover, any element in the ring gives a modular form with level structure. We get
 \begin{eqnarray}\label{mf}
  {M_{\Gamma_1(3)}}_*&\cong &\ZZ[1/3, a_1,a_3,\Delta^{-1}].
  \end{eqnarray}
 One should mention that all moduli problems for $\Gamma_1(n)$ and $\Gamma(n)$-structures with $n\geq 3$ are representable this way but $\Gamma_0(n)$ is not representable.
A good reference for these classical results on modular forms and level structures is the book \cite{MR772569}.\par
There is a derived version of these concepts as follows.
\begin{thm}\cite{MR2648680, HL13, AHR14}\label{tmfsheaf}
There is a sheaf ${\mathcal O}_{TMF}$ of $E_\infty$-ring spectra over ${\mathcal M}$ in the \'etale topology. This sheaf satisfies:
\begin{enumerate}
\item
The spectrum $TMF=\Gamma {\mathcal O}_{TMF}$ only has 2 and 3 torsion in homotopy. Away from 6 it is concentrated in even degrees and we have an isomorphism
$$ \pi_{2k}TMF[1/6]\cong M_k[1/6].$$
\item There is an orientation map $MString\lra TMF$ which induces the Witten genus in homotopy. In fact, its image coincides with the homotopy groups of a connective version of $TMF$. 
\item 
The sequence (\ref{level}) of moduli stacks gives a sequence of spectra
$$\xymatrix{ TMF(n)& TMF_1(n)\ar[l]& TMF_0(n)\ar[l]& TMF[1/n]\ar[l]}$$
and induced equivalences (with $n$ inverted)
\begin{eqnarray*}
TMF[1/n]&\cong &TMF(n)^{h \, Gl_2(\ZZ/n)}\\
TMF_0(n)&\cong &TMF_1(n)^{h \, Gl_1(\ZZ/n)}.
\end{eqnarray*}
\item
For all \'{e}tale $f: Spec( R ) \lra {\mathcal M}$ the spectrum $E=\Gamma f^*{\mathcal O}_{TMF}$ is a complex orientable ring spectrum whose formal group $E^0BS^1$ is equipped with an isomorphism to the formal completion of $C_f$. 
\end{enumerate} 
\end{thm}
The spectrum $TMF_1(3)$ can be described much more elementarily if one is only interested in its associated cohomology theory. Since the moduli problem is representable, its homotopy coincides with the ring $M_{\Gamma_1(3)}$ by property (iv) as we now explain. (The ring $M_{\Gamma_1(3)}$ has been determined in (\ref{mf}).) When we choose a coordinate on the formal group of the curve (\ref{curve}) we obtain a formal group law. These are classified by a map from the Lazard ring $L$, which coincides with the homotopy groups of complex bordism $MU$. For instance, the 2-typicalization of the standard coordinate has  Hazewinkel generators (see \cite[Lemma 1]{MR2076927}) at $p=2$
\begin{eqnarray}\label{v1}
v_1&=&a_1\\\label{v2}v_2&=&a_3.
\end{eqnarray} 
The Hazewinkel generators form a regular sequence and hence satisfy the Landweber exactness conditions. 
Thus 
for finite complexes $X$ we have natural isomorphisms
\begin{eqnarray}\label{LE}
TMF_1(3)^*X&\cong &MU^*X\otimes_{MU^*}M_{\Gamma_1(3)}.
\end{eqnarray}
(In  the older literature  for example  \cite{MR1071369,  MR1235295, MR1271552} this theory carried the names $Ell^{\Gamma_1(3)}$ or $E^{\Gamma_1(3)}$.)\begin{lemma}\label{e2}
The map 
$$\xymatrix{ \ZZ_{(2)}[v_1,v_2^{\pm 1},(v_1^3-27v_2)^{-1}] \ar[r]& {TMF_1(3)_{(2)}}_*}$$
is an isomorphism.
\end{lemma}
\begin{proof}
The discriminant of the universal curve has the form
$$ \Delta = a_3^3(a_1^3-27a_3).$$
The result follows from (\ref{v1}) and (\ref{v2}).
\end{proof}
Let $w_k\in \pi_*MU$ denote any element which projects to a generator of the indecomposables in $\pi_{2(p^k-1) }BP$. Recall from \cite[Definition 3.9]{Ma13} that an $MU$-module spectrum $M$ is a form of $BP\langle n\rangle  $ if there are $w_{n+1}, w_{n+2}, \ldots$ such that 
$$ M\cong BP/(w_{n+1}, w_{n+2}, \ldots).$$ In the sequel,  we will  call these  theories generalized  $BP\langle n\rangle $ and will use the same notation.
\begin{definition}
An $MU$-module spectrum $M$ is a generalized $E(n)$ if there exist a generalized $BP\langle n\rangle$ and an element $w\in \pi_*MU$ such that
$$  BP\langle n\rangle [w^{-1}] \cong M$$ and the element $w$  coincides with a power of $v_n$ modulo the ideal $(p,v_1,\ldots , v_{n-1})$.
\end{definition}
\begin{prop}\label{ge2}
Locally at the prime 2, the spectrum $TMF_1(3)$ is a generalized Johnson-Wilson theory $E(2)$.
\end{prop}
\begin{proof}
Use  \cite{MR2897051} or \cite[Example 3.11]{Ma13} to get a connective version of $TMF_1(3)$ which is a generalized  $BP\langle 2\rangle  $. The claim then follows from Lemma \ref{e2}.
\end{proof}
Theorem \ref{main} involves the elliptic character map. This map originates from  the Tate curve
$$ y^2+xy=x^3+B(q)x+C(q)$$
where 
$$ B(q)=\frac{-1}{48}(E_4(q)-1) \mbox{ and } \; C(q)=\frac{1}{496}(E_4(q)-1)-\frac{1}{864}(E_6(q)-1)$$
with the Eisenstein series $E_4$ and $E_6$. The series $B$ and $C$ are integral power series in $q$.
The evaluation of an ordinary modular form on the Tate curve with its canonical differential corresponds to its $q$-expansion. The formal group associated to the Tate curve is the multiplicative formal group. In order to get a $\Gamma_1(n)$-structure, that is, a point of order $n$,  one can use the extension of scalars $\ZZ((q))\rightarrow \ZZ((q)); \; q\mapsto q^n$. The resulting curve is usually denoted by $Tate(q^n)$ and its multiplicative reduction furnishes the Miller character map \cite{MR1022688}\cite{MR1660325}
$$\lambda: TMF_1(n) \lra K[1/n]((q)).$$
In homotopy this map is the classical $q$-expansion. 
\begin{lemma}\label{ff}
Locally at $p=2$, the map $$\lambda_*X: {TMF_1(3)}[a_1^{-1}]_*X \lra K[1/3]((q))_*X$$ is a monomorphism for all  $X$.
\end{lemma}
\begin{proof}
Note that the $q$-expansion  of $a_1$ starts with 1 and hence is invertible in the target. We will use Equation  (\ref{LE}) and replace $MU$ with $BP$.
Every comodule is the inductive limit of its finitely generated subcomodules. Hence, we may assume that  
 $BP_*X$ has a finite Landweber filtration $(F_k)$ with subsequent quotients of the form $BP_*/I_t$ with  $I_t$  the invariant prime ideal $(p,v_1,\ldots ,v_{t-1})$. By the $q$-expansion principle the character map $\lambda_*$ is injective and injective mod p. Hence it is injective when tensored with each of the quotients $BP_*/I_t$. The claim follows from the obvious inductive argument.
\end{proof}
\begin{remark}
It would be interesting to set up character maps for finer structures like $\Gamma_0(n)$. We will come back to this question in a subsequent work.
\end{remark}
\section{The $TMF_1(3)$-cohomology of $B{\mathit Spin}$}\label{proof thm}
In this section we will show the universal coefficient isomorphism for the $TMF_1(3)$-cohomology of the space $B{\mathit Spin}$ and the pullback diagram of Theorem \ref{main}. The main ingredient is the Morava $K(2)$-homology which has been computed by Kitchloo-Laures  in \cite{MR1909866}. We can use this result to obtain information about the $E(k,n)$-cohomology by methods of Ravenel-Wilson-Yagita. A chromatic argument enables us to compute the universal coefficients spectral sequence for $TMF_1(3)^*B{\mathit Spin}$. We will show that it collapses at the $E_2$-term and obtain  Theorem \ref{main} from the classical $q$-expansion principle. 
Some of the results of this section apply to other situations and hence are formulated more generally than actually needed. 
\par
From now on we fix a prime $p$. In a first step we do not want to deal with $\lim^1$-questions and hence work with $p$-completed spectra.  
Let $E(n)$ be a $p$-completed generalized Johnson-Wilson spectrum. Let $I_k$ be the invariant prime ideal $(p,v_1,\ldots ,v_{k-1})$ and let $E(k,n)$ be the spectrum $$E(k,n)=E(n)/I_k.$$ By definition we have 
$ E(n,n)\cong K(n)$, $E(0,n)=E(n)$ and there are cofibre sequences
\begin{eqnarray}\label{csekn}
\xymatrix@R=0.6pc{E(k-1,n)\ar[r]^{v_{k-1}} & E(k-1,n)\ar[r] & E(k,n) }.
\end{eqnarray}
\begin{remark}\label{ffe}
Recall from  Hovey and Sadofsky \cite[Theorem 3.4]{MR1722151} that for all generalized $K(n)$, there is a faithfully flat extension of its coefficient ring  over which its formal group law becomes strictly isomorphic to the Honda formal group law. This allows us to carry over results from the classical to the generalized $K(n)$.
\end{remark}

In the following, for a fixed $0<k\leq n$, let $X$ be a space with even $E(k,n)$-cohomology. The exact sequence induced by (\ref{csekn})
\begin{equation}\label{exseq}
\begin{aligned}
\xymatrix@R=0.6pc{0\ar[r]  &E(k-1,n)^{ev}X\ar[r]^{v_{k-1}} & E(k-1,n)^{ev}X\ar[r] & E(k,n)^{ev}X &\\ \ar[r]& E(k-1,n)^{odd}X\ar[r]^{v_{k-1}} & E(k-1,n)^{odd}X\ar[r]   &0}
\end{aligned}
\end{equation}
tells us that each element in $E(k-1,n)^{odd}X$ is infinitely divisible by $v_{k-1}$. The following result applies:
\begin{thm}\label{RWY}
If $x$ is infinitely divisible by $v_{k}$ in $E(k,n)^*X$, then it is zero.
\end{thm}
\begin{proof}
The proof for 
\cite[Corollary 4.11]{MR1648284} verbatim carries over to the generalized $E(k,n)$. In more detail,  one first observes that $x$ is infinitely divisible when restricted to each finite subcomplex of $X$. Hence one can assume that $X$ is finite. Then for $P(k)=BP/I_k$  one tensors a Landweber filtration of $P(k)^*X$ with the generalized $E(k,n)^*$ to reduce the claim to modules of the form $E(k,n)^*/(v_k, \ldots v_s)$ for some $s<n$. Here, $x$ has to be zero.
\end{proof}
\begin{cor}\label{even}
If $X$ is a space with even Morava $K(n)$-cohomology then $E(k,n)^*X$ is even for all $k$ and the exact sequences (\ref{exseq}) are short exact. 
\end{cor}
\begin{proof}
By an inductive argument we may assume that all elements of $E(k-1,n)^{odd}X$ are infinitely divisible by $v_{k-1}$. Then the claim follows from Theorem \ref{RWY}.
\end{proof}
Next we look at the universal coefficient spectral sequence for $E=E(n)$
$$\Ext_{E_*}^{s,t}(E_*X, E_*)\Longrightarrow E^{s+t}X$$

\begin{lemma}\label{gl-dim}
The global dimension of  $E_*=E(n)_*$ in the category of graded modules equals $n$.
\end{lemma}
\begin{proof}
The corresponding result is well known in the ungraded setting. Its graded version is harder to find in the literature  but the proof given in \cite[19.5]{MR1322960}  carries over:  Let $k_*$ be the graded field ${\mathbb F}_p[v_n^{\pm}]$. Since $(p,v_1,\ldots ,v_{n-1})$ is a regular sequence the Koszul complex provides a free graded resolution of length $n$. This implies the vanishing of $\Tor^{E_*}_{i+1}(k_*,M_*)$ for all $i\geq n$ and for all $M_*$. Next let $F=(F_{n *},\varphi_n)$  be a graded minimal free resolution of of a finitely generated module $M_*$. (The minimality condition means that for each $n$ a basis of $F_n$ maps to a minimal set of generators of $\ker   \varphi_{n-1}$.) Since all differentials in $k_*\otimes F$ are 0 we have 
$$ \Tor^{E_*}_{i+1}(k_*,M_*)\cong k_*\otimes_{E_*}F_{i+1,*}.$$
This vanishes iff $F_{i+1,*}$ vanishes because the resolution is free. Thus we have shown the claim for finitely generated modules. The general  result follows from the graded version of Auslander's Theorem \cite[19.1]{MR1322960}.
\end{proof}

\begin{lemma}\label{target}
Let $E$ be $E(2)$  and let $E/p^{\infty}$ be the cofibre of $E\ra p^{-1}E$. Suppose $\Ext^{0,t+1}_{E_*}(E_*X,E_*/p^\infty)$ vanishes for all even $t$.
Then for even $t$ there is an isomorphism
$$ \Ext^{2,t}_{E_*}(E_*X,E_*)\cong (E/p^\infty)^{t+1} X .$$

\end{lemma}
\begin{proof} 

Suppose $F$ is a ring theory with rational coefficients. Then Serre's result and the comparison theorem for homology theories imply that  the $F_*$-linear extension of the $F$-Hurewicz map
$$ \xymatrix{F_* \otimes \pi_*^{st}(X) \ar[r]& F_*(X)}$$
is an isomorphism.  It follows  that   $p^{-1}E_*X$ is a free $p^{-1}E_*$ module. Hence
 $$  \Ext^{s}_{E_*}(E_*X,p^{-1}E_*)\cong  \Ext^{s}_{p^{-1}E_*}(p^{-1}E_*X,p^{-1}E_*)$$
vanishes for all $s>0$, and
the short exact sequence
$$\xymatrix{E_*\,\,  \ar@{>->}[r]& p^{-1}E_* \ar@{->>}[r]&E_*/p^{\infty}}$$ induces isomorphisms
$$  \Ext^{s+1}_{E_*}(E_*X,E_*)\cong  \Ext^{s}_{E_*}(E_*X,E_*/p^{\infty}).$$
By \ref{gl-dim} these groups vanish for $s\geq 2$ and so does $\Ext_{E_*}^{0,t+1}(E_*X,E_*/p^{\infty})$  for even $t$. 
Hence the claim follows from the universal coefficient spectral sequence
$$\Ext_{E_*}^{s,t}(E_*X, E_*/p^\infty)\Longrightarrow  (E/p^\infty)^{s+t} X.$$
\end{proof}

We now specify to the case $X=B{\mathit Spin}$ and  prove a splitting principle using the following computation.
\begin{thm}\cite[1.2]{MR1909866}
Let $n$ be 1 or 2. Let  $ b_i \in K(n)_{2i}BS^1$ be the dual to the power ${c_1^i}$ of the first Chern class of the canonical line bundle. Denote its image  under the map induced by the inclusion of the maximal torus
$$\xymatrix{BS^1\ar[r]& B{\mathit Spin}(3)\ar[r]&B{\mathit Spin}}$$
by the same name. 
Then we have
$$
K(n)_*B{\mathit Spin} \cong \pi_*K(n)[b_{2^n \cdot 2}, b_{2^n \cdot 4},b_{2^n \cdot 6},\ldots].
$$
\end{thm}
\begin{cor}
Let $T(\lfloor m/2\rfloor )$ be the maximal torus of $Spin(m)$ and set $ BT^\infty=\mbox{co}\!\lim BT(\lfloor m/2\rfloor )$. Then  the restriction map from $K(n)^*B{\mathit Spin}$ to $K(n)^*BT^{\infty}$ is injective for $n\leq2$.
\end{cor}
\begin{proof}
It suffices to show that the dual map
$K(n)_*BT^{\infty}\lra K(n)_*B{\mathit Spin} $ is surjective. This is immediate from the theorem since each monomial  $b_{i_1}b_{i_2}\cdots b_{i_k}$ comes from the classifying space of the  $k$-dimensional torus.
\end{proof}
\begin{cor}\label{torus}
For $E=E(k,n)$ with $n\leq 2$ the restriction map from $E^*B{\mathit Spin} $ to $E^*BT^{\infty}$ is injective.
\end{cor}
\begin{proof}
By descending induction on $k\leq n$, by the previous corollary and by Corollary \ref{even} we have a map of short exact sequences
$$ \xymatrix{E(k-1,n)^*B{\mathit Spin}\, \, \ar@{>->}[r]^{v_{k-1}}\ar[d]&
E(k-1,n)^*B{\mathit Spin}  \ar@{->>}[r]\ar[d]&E(k,n)^*B{\mathit Spin}\ar@{>->}[d]\\
E(k-1,n)^*BT^\infty\, \, \ar@{>->}[r]^{v_{k-1}}&E(k-1,n)^*BT^\infty\ar@{->>}[r]&E(k,n)^*BT^\infty}
$$
for which the last vertical map can be assumed to be injective.
Hence, any element in $E(k-1,n)^*B{\mathit Spin}$ which restricts trivially to the torus is divisible by $v_{k-1}$ and the quotient  again restricts trivially to the torus. Continuing this way, we see that it must be infinitely divisible by  $v_{k-1}$ and thus has to vanish by Theorem \ref{RWY}.
\end{proof}
\begin{thm}\label{UCTE(k)} 
Let $E$ be $E=E(1)$ or the 2-completed $TMF_1(3) $. Then  the universal coefficient isomorphism
$$ E^*B{\mathit Spin} \cong \Hom_{E_*} (E_*B{\mathit Spin}, E_*)$$ holds.
\end{thm}
\begin{proof}
First note that we have the isomorphism
$$ E^*BT^\infty\cong \Hom_{E_*} (E_*BT^\infty,E_*)$$
because $E_*BT^\infty\cong \colim E_*T^k$ is free: a basis is given by arbitrary products of the form
$$\beta_{i_1}\otimes \beta_{i_2}\otimes \cdots \otimes \beta_{i_k}$$ 
where $\beta_i$ is dual to $c_1^i$.\par
For the space $X=B{\mathit Spin}$ and for  $E=E(1)$ the universal coefficient spectral sequence degenerates to the short exact sequence
$$\xymatrix{\Ext^{1,*}_{E_*}(E_{*-1}X,E_*)\,\,  \ar@{>->}[r]& E^*X\ar@{->>}[r]&\Ext^{0,*}_{E_*}(E_{*}X,E_*)}$$
by Lemma \ref{gl-dim}.
Furthermore, the second map is injective because it factors through the restriction map into  $ E^*BT^\infty\cong \Ext^{0,*}_{E_*} (E_*BT^\infty,E_*)$ and this map is injective by Corollary  \ref{torus}. \par

For the case $E=TMF_1(3)$ at $p=2$, we already know from Corollary \ref{even} that $E^*X$ is concentrated in even degrees. Since $E_*$ is torsion free and $p^{-1} E_*X$ is in even degrees, the group $\Hom_{E_*}(E_*X,E_*)$ must be even, too. Hence, we can restrict our attention to even degrees. 
\par
We would like to apply Lemma \ref{target}, and therefore we have to show that $\Ext^{0,t+1}_{E_*}(E_*X,E_*/p^\infty)$ vanishes for all even $t$. By Lemma \ref{ff} or simply by the $q$-expansion principle, the group $E_*/p^\infty$ injects into ${K_{Tate}}_*/p^\infty$ and so does the induced map on $\Ext^0$-groups. The isomorphisms 
 at $p=2$
\begin{eqnarray*}
\Ext^{0,t+1}_{E_*}(E_*X,{K_{Tate}}_*/p^\infty)&\cong& \Ext^{0,t+1}_{BP_*}(BP_*X,{K_{Tate}}_*/p^\infty)\\ &\cong &\Ext^{0,t+1}_{E(1)_*}(E(1)_*X,{K_{Tate}}_*/p^\infty)
\end{eqnarray*}
follow from the Landweber exactness of $E$ and $K_{Tate}$. The latter group   coincides with  
$\Ext^{1,t+1}_{E(1)_*}(E(1)_*X,{K_{Tate}}_*)$ which is 
 a product of groups of the form $\Ext^{1,t+1}_{E(1)_*}(E(1)_*X,E(1)_*)$.  Thus it is trivial.
\par
The even dimensional $\Ext^2$-term of the  universal coefficient spectral sequence has been  identified with the odd part of  $({E/p^\infty})^* X$ in Lemma \ref{target}. Hence, for its vanishing it is enough to show the injectivity of the first map in the exact sequence
$$  \xymatrix{E^*X \ar[r]& (p^{-1}E)^*X \ar[r]& (E/p^\infty)^*X}.$$
This follows from the observation that its composite with the restriction map  to $(p^{-1}E)^* BT^\infty$ is injective by Corollary \ref{torus}. In even degrees we obtain the  short exact sequence 
$$\xymatrix{\Ext_{E_*}^{1,*}(E_{*-1}X,E_*)\, \, \ar@{>->}[r]& E^*X\ar@{->>}[r]&\Ext_{E_*}^{0,*}(E_{*}X,E_*)}.$$
Once more  by Corollary \ref{torus} the kernel of the second map vanishes.
\end{proof}
\begin{remark}
We do not know much about the generalized $E(2)$-homology of $B{\mathit Spin}$, not even if it is concentrated in even degrees. This is what makes the proof of the universal coefficient isomorphism difficult.
\end{remark}

 \begin{proof}[Proof of Theorem \ref{main}:] 
 We first show the pullback property for the  $2$-completed theories. Recall from \cite[Lemma 1.5]{MR1660325}  that there is a pullback of rings
 $$ \xymatrix{ TMF_1(3)_*\ar[r]\ar[d]& K[1/3]_*((q))\ar[d]\\
 (TMF_1(3)_*)_\QQ\ar[r]& (K_*)_\QQ((q))}
 $$
 and each corner is the coefficient ring of a Landweber exact theory. For each such theory $E$, and for all spaces $X$, we have the natural isomorphism
  $$\Hom_{BP_*}(BP_*X, E_*)\cong  \Hom_{E_*}(E_*X, E_*).$$
 
 Hence, when applying the left exact functor $\Hom_{BP_*}(BP_*B{\mathit Spin}, -)$ to the diagram we still have a pullback. By  Theorem \ref{UCTE(k)} each  corner satisfies the universal coefficient isomorphism and hence we get the desired pullback diagram.\par
 It remains to show the integral result. For odd primes we still have the appropriate pullback diagram by \cite[Theorem 1.12]{MR1660325}. Moreover, for all spectra $Y$ we have an arithmetic pullback by \cite[Proposition 2.9]{MR551009}
 \begin{equation} \label{as}
\begin{aligned}
 \xymatrix{ Y \ar[r]\ar[d]& \prod_{p}L_{S\ZZ /p}Y\ar[d]\\
 L_{S\QQ}Y\ar[r]&  L_{S\QQ}(\prod_{p}L_{S\ZZ /p}Y )}
 \end{aligned}
 \end{equation}
 where $SG$ denotes the Moore spectrum.
Mapping $X=B{\mathit Spin} $ to the diagram for $Y=TMF_1(3)$ yields a pullback of cohomology groups by the previous results. In this pullback diagram the lower right corner has the form
$$H^*(X, \left(\prod_p {TMF_1(3)}_p^\wedge \right)^*_\QQ)\cong \left(\prod_p TMF_1(3)^\wedge_p\right)^*_\QQ X,$$
which we can replace with
$$ H^*(X, \prod_p \left({TMF_1(3)}_p^\wedge\right)^*_\QQ)\cong \prod_p\left(TMF_1(3)^\wedge_p\right)^*_\QQ X.$$
Then the enlarged  diagram 
 $$ \xymatrix{ TMF_1(3)^*X  \ar[r]\ar[d]&\prod_p\left(TMF_1(3)^\wedge_p\right)^* X\ar[r]\ar[d]& \prod_{p}\hat{K}_p^*X((q))\ar[d]\\
 TMF_1(3)^*_\QQ X\ar[r]& \prod_p\left(TMF_1(3)^\wedge_p\right)^*_\QQ X\ar[r]&\prod_{p}(\hat{K}_p^*)_\QQ X ((q)) }$$
 is a pullback.
Since the upper composite factors through $K_{Tate}^*X $ the result follows.
 \end{proof}

 \begin{cor}\label{UCT global}
The integral universal coefficient isomorphism
\begin{eqnarray*}
TMF_1(3)^*B{\mathit Spin} &\cong & \Hom_{TMF_1(3)_*}(TMF_1(3)_*B{\mathit Spin}, TMF_1(3)_*)
\end{eqnarray*}
holds.
\end{cor}
\begin{proof} The right hand side satisfies the pullback property in the diagram of the main theorem.
\end{proof}
\section{Pontryagin classes and the cohomology of $B{\mathit String}$}\label{SecBString}
In this section we construct explicit generators in the $TMF_1(3)$-cohomology rings of $B{\mathit Spin}$ and $B{\mathit String}$ with the help of Theorem \ref{main}.
We start with a reminder of the $KO$-Pontryagin classes. We will see that for line bundles they are defined by the same formula as the Pontryagin classes in singular homology. Using Theorem \ref{main} we then obtain a construction of $TMF_1(3)$-Pontryagin classes. \par
Recall from \cite[Proposition 4.4(b)]{MR0189043}  that the restriction maps 
from $ KO^0(B{\mathit SO}(n))$ to the maximal torus $K^0(BT(\lfloor n/2\rfloor ))$ are monomorphisms with images the invariants of the Weyl groups. The $KO$-Pontryagin classes are defined  by the preimage of the series
 \begin{eqnarray}\label{Pseries}
  \prod_{i=1}^{\lfloor n/2\rfloor } (1-t(x_i \overline{x}_i)) \in K^0(BT(\lfloor n/2\rfloor )).
  \end{eqnarray}
 Here, $x_i=1-L_i$ are the first Chern classes of the canonical line bundles $L_i$ over $BT(\lfloor n/2\rfloor)$ in  $K$-theory. 
\par 
The Pontryagin classes freely (topologically) generate the ring $KO^0(B{\mathit SO})$. Since we do not know a reference for this fact we give a short argument: first note that in $K$-theory we have for all $x=1-L$ the equality
 $$x+\bar{x}=x\bar{x}.$$
Hence a power series  which is invariant under the map which interchanges $x$ and $\bar {x}$ can be written as a power series in $x\bar{x}$. We conclude that each class in $KO^0(B{\mathit SO})$ is a symmetric power series in $x_i\bar{x_i}$ when restricted to $K^0BT$ for all tori $T$. Thus it is a power series in the Pontryagin classes. 
\par
We also note that
$$  KO^0(B{\mathit Spin})\cong KO^0(B{\mathit SO}).$$
This fact follows from the arithmetic square   (\ref{as}) since the map from $B{\mathit Spin}$ to $B{\mathit SO}$ is a $K(1)$-local and rational equivalence (see \cite[Theorem 1.2(ii)]{MR1909866}).

 \begin{proof}[Proof of Theorem \ref{Pontryagin}: ]
 The $K$-Pontryagin classes freely generate the $K^*_{Tate}$-algebra $K^*_{Tate}(B{\mathit Spin})$. We also know that the classical Pontryagin classes in rational singular cohomology freely generate $H^*(B{\mathit Spin}; TMF_1(3)^*_{\QQ})$ as a $TMF_1(3)^*_{\QQ}$-algebra.  They are defined in the same way except that in Formula (\ref{Pseries}), the $x_i$ are the ordinary first Chern classes. 
 \par
 The multiplicative formal group law over ${K_{Tate}}_*$ is strictly isomorphic to the one coming from the Tate curve, that is, the $q$-expansion of the curve (\ref{curve}) together with the standard coordinate. By  \cite{MR1022688}  there is a natural automorphism of $K_{Tate}$ which exchanges the corresponding  two orientations. Hence, when we replace the $x_i$ in the formula for the $K$-Pontryagin classes by the first Chern classes with respect to the new orientation we still have free generators. The same argument holds for rational singular homology because here all formal group laws are strictly isomorphic. \par
These Pontryagin classes define elements in each corner of the pullback diagram of Theorem \ref{main} and hence free generators of $TMF_1(3)^*B{\mathit Spin}$. Moreover, these classes are determined by their restrictions to the maximal tori by Corollary \ref{torus} and are here given by the displayed formula. \end{proof}
Next we consider the cohomology of $B{\mathit String}$.
The space $B{\mathit String}$ is defined as the homotopy fibre of the map
$B{\mathit Spin} \rightarrow K(\ZZ, 4)$ which kills the lowest homotopy group. In particular, we have a sequence of infinite loop spaces
\begin{eqnarray}\label{eqn1} &\xymatrix{K(\ZZ/2,2)\ar[r]& K(\ZZ ,3)\ar[r]& B{\mathit String} \ar[r]&B{\mathit Spin}}.
\end{eqnarray}
\begin{thm}[\cite{MR2093483} \cite{MR2086079}]\label{KLW}
The sequence (\ref{eqn1}) induces an exact sequence of Hopf algebras in Morava $K(2)$-homology at the prime 2
$$\xymatrix@R=0.6pc{K(2)_* \ar[r]  &K(2)_*K(\ZZ/2,2)\ar[r] & K(2)_*K(\ZZ ,3)\ar[r] & K(2)_*B{\mathit String} &\\ \ar[r]& K(2)_*B{\mathit Spin}\ar[r]  &K(2)_*.}
$$
Algebraically, there is short exact sequence of Hopf algebras
$$\xymatrix{ K(2)_*K(\ZZ ,3)\ar[r] & K(2)_*B{\mathit String}  \ar[r]& K(2)_*B{\mathit Spin}}
$$
which splits. (Here, the first map is not the one which is induced from the second map of (\ref{eqn1}).)
In particular, the module $K(2)_*B{\mathit String}$ is concentrated in even dimensions. 
\end{thm}
We will use this information for the computation of the cohomology ring of $B{\mathit String}$ with respect to 
$$E=L_{K(2)}E(2)$$
for a generalized $E(2)$.
The coefficients of $E$ are given by 
$E_*=(E(2)_*)^{\wedge}_{I_2}$.
\begin{prop}[\cite{MR1601906} Proposition 2.5]\label{HovStrick}
Suppose $X$ is a space with even Morava $K(2)$-homology.  Then $E^*X$  is the completion with respect to $I_2$ of a free $E_*$-module.
\end{prop}
\begin{prop}\label{K(Z,3)}
There is an isomorphism of algebras
$$E^*K(\ZZ,3) \cong E^* [\! [r]\!] $$
with $r$ in degree 6.
\end{prop}
\begin{proof}
The $K(2)$-cohomology of $K(\ZZ,3)$ has been computed in \cite{MR584466} and \cite{MR784291} (see also Su \cite{MR2709571} for its $E(1,2)$-cohomology.) It is topologically free on a generator of degree 6.  Lift this generator to an $E$-cohomology class. This is possible because the algebra $E^*K(\ZZ,3)$ is concentrated in even degrees with the Morava $K(2)$-cohomology as its $I_2$-reduction. (This can be seen as before with the exact sequence   (\ref{exseq}).)
Clearly, when restricted to a finite subcomplex of $K(\ZZ,3)$ every reduced class becomes  nilpotent. Hence, we obtain an algebra map  from $ E^* [\! [r]\!]$ to $E^*K(\ZZ,3)$. The result follows from Proposition \ref{HovStrick}  and the following version of Nakayama's Lemma.
\end{proof}
\begin{lemma} Let $R$ be a graded ring with a unique maximal homogeneous ideal $m$ and let $P$ and $Q$ be pro-free $R$ modules. 
Then  $f:P \ra Q$  is an isomorphism  if and only if it is so modulo $m$.
\end{lemma}

\begin{proof}
Tensor the short exact sequences
$$\xymatrix{m^k/m^{k+1}\, \, \ar@{>->}[r]& R/m^{k+1}\ar@{->>}[r]& R/m^k}$$
with $f$ and use the fact that $m^k/m^{k+1}$ is a free module over $R/m$. 
\end{proof}

\begin{proof}[Proof of Theorem \ref{BString}: ]
By Theorem \ref{KLW}, we find a lift  $r$ of the generator of $K(2)^*K(\ZZ,3)$ to $K(2)^*B{\mathit String}$. Since  $K(2)^*B{\mathit String}$ is concentrated in even degrees it is the quotient of $E^*B{\mathit String}$ by the ideal $I_2$. Hence we can lift $r$ further to a class in $E^*B{\mathit String}$ and obtain an algebra map
$$f:\xymatrix{ E^*B{\mathit Spin} \hat{\otimes }E^* [\! [r]\!]\ar[r] &E^*B{\mathit String}}.$$
This map  is an isomorphism by Proposition \ref{HovStrick} and the previous lemma. \end{proof}
\begin{remark}
Let $E$ be  the 2-complete $TMF_1(3)$. Then there is a pullback square of cohomology rings 
 $$ \xymatrix{E^*B{\mathit String}\ar[r]\ar[d]& (L_{K(2)}E)^*B{\mathit String}\ar[d]\\ (L_{K(1)}E)^*B{\mathit String}\ar[r]& (L_{K(1)}L_{K(2)}E)^*B{\mathit String}}$$
for which we have computed the three corners: $K(1)$-locally the map from $B{\mathit Spin}$ to $B{\mathit String}$ is an equivalence because the $K(1)$-homology of $K(\ZZ,3)$ vanishes.  Hence, the lower horizontal map is the inclusion
$$\xymatrix{ (L_{K(1)}E)^*[\! [p_1,p_2,\ldots ]\!]\ar[r]&(L_{K(1)}L_{K(2)}E)^*[\! [p_1,p_2,\ldots ]\!]}$$
but for the right vertical map the image of the class $r$  is unclear. We will address this question in a subsequent work. 
\end{remark}

\section{Applications to representations of loop groups}
A principal  $Spin(d)$-bundle $P$ over $X$ and a representation $V$ of $Spin(d)$ give rise to a vector bundle  over $X$  by associating $V$ to each fibre of $P$.  Hence $P$ and  $V$ define an element in the $K$-theory ring $K^0(X)$. If $X$ is a spin manifold and $P$ is the principal bundle associated to its tangent bundle, the pushforward of this $K$-theory class  to a point is the index of the Dirac operator twisted by $V$ (see \cite{MR0236950}). The construction induces a ring map
\begin{eqnarray}\label{RSpin}
\xymatrix{Rep(Spin(d)) \ar[r] & K^0(X)}.
\end{eqnarray}
It factors through the map from $K^0(B{\mathit Spin}(d))$ to $K^0(X)$ which classifies $P$ (see \cite{MR0259946}). \par
In \cite{MR1071369} Brylinski conjectured a similar connection for loop space representations and the elliptic cohomology  of the 7-connected cover of $B{\mathit Spin}(d)$ which is usually denoted by $B{\mathit String}(d)$. Recall from \cite{MR900587}) that
 a positive energy representation $V$ of the free loop space $L{\mathit Spin}(d)$  means that $V$ is a representation of the semi-direct product $\tilde{L}Spin(d)\rtimes S^1$ where $\tilde{L}Spin(d)$ is the universal central extension of $L{\mathit Spin}(d)$ and $S^1$ acts on $\tilde{L}Spin(d)$ by reparameterizing the loops. `Positive energy' means that  for the action of $R_t=e^{2\pi it} \in S^1 $ the vector space 
$$V_n=\{ v\in V:\: R_t(v)=e^{2\pi in t}v \mbox{ for  all }t\}$$
is finite dimensional for all $n$ and vanishes for all $n< n_0$ for some $n_0$. (Note that we can always multiply the rotation action with a character of $S^1$ to get $n_0=0$.) When identifying the group $Spin(d)$ with a subgroup of $\tilde{L}Spin(d)\rtimes S^1$, we obtain a formal power series
$$\sum_{n\in \ZZ } [V_n] q^n\in Rep(Spin(d))((q))$$
and hence via the map (\ref{RSpin}) a class in $K_{Tate}^0(X)$ for every $Spin(d)$-bundle $P$ over X. If $P$ comes from a string bundle then the loop space $LX$ carries a $\tilde{L}Spin(d)$-principal bundle (see ibidem). Hence, the class in $K_{Tate}^0(X)$ associated to $V$ can be viewed as 
$$ (\tilde{L}P \times_{\tilde{L}Spin(d)}V)_{|X}.$$
In case $X$ is a string manifold and $P$ comes from its tangent bundle,  the pushforward of this $K_{Tate}$-class to a point  is the formal index of the $S^1$-equivariant Dirac operator on the free loop space of $X$ with coefficient in the bundle associated to the representation $V$ (see \cite{MR970288}).
The fact that this index is a modular form of some level
 leads to the hope that the $K_{Tate}(M)$-cohomology class refines to a class in topological modular forms (c.\cite{MR1071369}). 
More precisely, let $m$ be an integer and let $P_m(d)$ be the free abelian group generated by the isomorphism classes of irreducible positive energy representations of  $\tilde{L}Spin(d)$ of level $m$ (see \cite[9.3]{MR900587} for the meaning of `level' ). 
\begin{conj}\label{conj}
There is an integer $n$ depending on $d$ and $m$ and an additive map
$$\varphi_d: P_m(d)\lra TMF(n)^0B{\mathit String}(d)$$
whose elliptic character $\lambda$  (defined as before in \cite{MR1022688})  coincides with  the bundle 
\begin{eqnarray}\label{varphid}
 \lambda \varphi_d(V)&=&(\tilde{L} E{\mathit Spin}(d)\times_{\tilde{L}Spin(d)}V)_{|B{\mathit String}(d)}.
\end{eqnarray}
 \end{conj}
Let $n=24(m+g)$, with $g=d-2$ the Coxeter number. Theorem A in  \cite{MR954660} states that the character $J$ of a positive energy representation $V$ is a formal Jacobi modular form of level $n$, weight $w=0$ and index $\frac{m}{2}$. This means that it is  invariant under the action 
of $\Gamma(n)$ defined by
\begin{eqnarray}\label{gnaction}\qquad
\left(J \, \begin{bmatrix}a & b\\ c&d\end{bmatrix}\right)(z,\tau )=(c\tau +d)^{-w} e^{-2\pi i m c(\sum_j z_j^2)/(c\tau +d )} J(\frac{z}{c\tau +d}, \frac{a\tau +b}{c\tau +d})
\end{eqnarray}
when suitably normalized with a character of the rotation group. (In this formula the Chern roots $x_i$ are replaced by $2\pi i z_i$ (c.\cite{MR1071369}).) Since for string bundles the first Pontryagin class $p_1=\sum_j z_j^2$ vanishes we have the following result. 
\begin{thm}\cite[Remark 3.10]{MR1071369}\label{coxeter}
Conjecture \ref{conj} holds rationally for $n=24(m+g)$. 
\end{thm}
We will show a stable version of the integral conjecture. 
Let $P_m$ be the inverse limit of all $P_m(d)$'s. For $n$ divisible by 3,  let $(P_m)_{\Gamma(n)}$ be the subgroup of $P_m$ consisting of  representations $V$ with character $J$ invariant under the action (\ref{gnaction}) of $\Gamma(n)$.
\begin{thm}\label{loopthm}
There is 
an additive map
$$\varphi: (P_m)_{\Gamma(n)}\lra \widehat{TMF(n )}^0B{\mathit String}$$
whose elliptic character is the class described in Equation (\ref{varphid}). 
\end{thm}
The proof will be given at the end of the section. 
\begin{lemma}\label{change level}
Let 3 be inverted in the following rings of modular forms and let $n$ be divisible by 3. Then the following ring extensions are flat:
\begin{eqnarray*}
M_{\Gamma_1(3)}&\lra &  M_{\Gamma(3)}\\
M_{\Gamma(3)}&\lra & M_{\Gamma(n)}
\end{eqnarray*}
\end{lemma}
\begin{proof}
This follows from \cite[5.5.1]{MR772569} since the moduli problem for $\Gamma_1(3)$ structures is representable.
\end{proof}
\begin{lemma}\label{tensorhom}
Let $R$ be coherent, $M$ be a finitely generated $R$ module  and $N$ be a flat $R$-module. Then we have the isomorphism
$$ \Hom_R(M,R)\otimes_R N\cong \Hom_R(M,N).$$
\end{lemma}
\begin{proof}
This result should be standard. Choose a finitely generated free presentation $F_*$ of $M$. Since $N$ is flat the left hand side of the claim  is the kernel of 
$$ \Hom_R(F_0,R)\otimes_R N\lra \Hom_R(F_1,R)\otimes_R N .$$
By finiteness it coincides with the kernel of 
$$   \Hom_R(F_0,N)\lra \Hom_R(F_1,N).$$
which is the right hand side.
\end{proof}
\begin{prop}\label{tmfnpull}
The diagram 
$$
\xymatrix{TMF(n)^*B{\mathit Spin} \ar[r]^\lambda \ar[d]&K_{Tate}^*B{\mathit Spin}\ar[d]\\
 H^*(B{\mathit Spin}, TMF(n)^*_\QQ)\ar[r]&  H^*(B{\mathit Spin}, {K_{Tate}}^*_\QQ)}
 $$
 is a pullback for all $n$ divisible by 3.  
 \end{prop}
 \begin{proof}
 This follows as before from the universal coefficient isomorphism:
 \begin{eqnarray*}
TMF(n)^*B{\mathit Spin}&\cong &TMF_1(3)^*B{\mathit Spin} \hat{\otimes}_{M_{\Gamma_1(3)}}M_{\Gamma(n)}
\\ & \cong & \Hom(TMF_1(3)_*B{\mathit Spin}, M_{\Gamma_1(3)})\hat{ {\otimes}}_{M_{\Gamma_1(3)}}M_{\Gamma(n)}
\\ &\cong & \Hom(TMF_1(3)_*B{\mathit Spin}, M_{\Gamma(n)}).
\end{eqnarray*} 
Here, the first  isomorphism holds for all finite spectra by  Lemma \ref{change level}. (The completion of the tensor product is then unnecessary). Taking inverse limits over all finite subspectra and using  Corollary  \ref{UCT global} we see that the isomorphism holds for $B{\mathit Spin}$ for the completed tensor product. The second isomorphism follows again from Corollary \ref{UCT global}. The last one is a consequence of Lemmas \ref{change level} and  \ref{tensorhom} when considering finite subspectra of $B{\mathit Spin}$.
\end{proof}
\begin{proof}[Proof of \ref{loopthm}. ] Set $E=TMF_1(3)$.
We first construct a specific invertible element $S$ in $E^0B{\mathit String}$ whose character is a Jacobi form of weight 0 and index $\frac{1}{2}$.  Let $x: MU\ra E$ be the orientation considered earlier. Its exponential is given in terms of the Weierstrass $\Phi$-function  by
$$ \frac{\Phi(\tau , z )\Phi(\tau,-\omega)}{\Phi(\tau,z-\omega)}$$
for the standard division point $\omega=2\pi i/3$ (see \cite[5.3, 6.4]{MR1189136}). 
Let $BU\langle 6\rangle $ be the 5-connected cover of $BU$ and   $MU\langle 6\rangle $ be the corresponding Thom spectrum. Then $E$ admits two ring maps from  $MU\langle 6\rangle $: the one which factors through $x$ and the one which factors through the Witten orientation (Theorem \ref{tmfsheaf}(ii)). Using the Thom isomorphism we obtain a class
$$\frac{\sigma}{x}\in  E^0 BU\langle 6\rangle$$
whose augmentation is 1. Let $S$ be the image of $\frac{\sigma}{x}$ under the complexification  map  from $B{\mathit String}$ to $BU\langle 6\rangle $.  The character of $S$ is the function
$$ \frac{\Phi(\tau,z-\omega)}{\Phi(\tau,-\omega)}.$$ One can check that
this is a formal Jacobi function of weight 0 and index $\frac{1}{2}$ either by direct calculation or one uses the fact that the Witten orientation  $\sigma$ comes from a formal Jacobi modular form of index $\frac{1}{2}$ and $x$ is one of index 0  (see \cite[p.469]{MR1071369}). Since $\frac{\sigma}{x}$ has the inverse $\frac{x}{\sigma}$ the class $S$ is invertible.\par
Next, let 
$V$ be an irreducible positive energy representation in $(P_m)_{\Gamma(n)}$ and let 
$$\hat{V}= \sum_n [V_n]q^n$$   be the the class in $K_{Tate}^0(B{\mathit Spin})$ constructed above. 
Without loss of generality let $E$ be $L_{S\ZZ/2}TMF(n)$. Consider $S$ as a class in $E^0B{\mathit String}$. Since the $K(1)$-homology of $K(\ZZ,3)$ vanishes (c.\cite{MR584466}) the 2-completed $K_{Tate}$-cohomologies of  $B{\mathit String}$ and $B{\mathit Spin}$ coincide. Hence, we may consider the element 
$$\hat{V}\lambda(S)^{-m}\in K_{Tate}^0(B{\mathit Spin}).$$ Since its character is a Jacobi form of index 0 and weight 0 we can apply Proposition \ref{tmfnpull} to obtain a unique class in $E^0B{\mathit Spin}$ and hence in $E^0B{\mathit String}$. Define $\varphi[V]$ as the product of this class with $S^m$.

\end{proof}
\bibliographystyle{amsalpha}

\bibliography{toda}
\end{document}